\documentclass[12pt]{article}

\usepackage{amsmath,amssymb,bm,amsthm,mathrsfs}
\newtheorem{Def}{Def}[section]

\newtheorem{Them}[Def]{Theorem}
\newtheorem{Lem}[Def]{Lemma}
\newtheorem{Cor}[Def]{Corollary}
\newtheorem{Prop}[Def]{Proposition}
\newtheorem{Examp}[Def]{Example}

\numberwithin{equation}{section}

\newcommand{\Sym}{\operatorname{Sym}}

\newcommand{\DC}{\mathcal{D}}
\newcommand{\DS}{\mathscr{D}}

\newcommand{\NB}{\mathbb{N}}
\newcommand{\AC}{\mathcal{A}}

\newcommand{\ZB}{\mathbb{Z}}

\newcommand{\AS}{\mathscr{A}}
\newcommand{\BC}{\mathcal{B}}
\newcommand{\RB}{\mathbb{R}}

\title
{Modules of differential operators of order 2 on Coxeter arrangements}
\author{Norihiro Nakashima}
\date{}

\begin{document}

\maketitle

\begin{abstract}
The collection of reflection hyperplanes of a 
finite reflection group is called a Coxeter arrangement. 
A Coxeter arrangement is known to be free. 
K. Saito has constructed a basis consisting of 
invariant elements for the module of 
derivations on a Coxeter arrangement. 
We study the module of $\AS$-differential operators 
as a generalization of the study of
the module of $\AS$-derivations.
In this article, we prove that the modules of differential operators 
of order $2$ on Coxeter arrangements of types A, B and D are free, 
by exhibiting their bases. 
We also prove that the modules cannot have bases consisting of 
only invariant elements. 
Two keys for the proof of freeness are ``Cauchy-Sylvester's theorem 
on compound determinants'' and ``Saito-Holm's criterion.'' 
\vspace{5mm}
\\
{\bf Key Words:}\quad 
Coxeter arrangement, Cauchy-Sylvester's compound determinants, 
Schur functions.
\vspace{2mm}
\\
{\bf 2010 Mathematics Subject Classification:}
\quad Primary 32S22, Secondary 15A15.
\end{abstract}

\section{Introduction}
Let $V=\RB^{\ell}$ be a Euclidean space of dimension $\ell$ over $\RB$. 
Let $\{x_{1},\dots,x_{\ell}\}$ be a basis for the dual 
space $V^{\ast}$, and let $S:=\Sym(V^{\ast})\simeq \RB[x_{1},\dots,x_{\ell}]$ 
be the polynomial ring. Put $\partial_{i}:=\frac{\partial}{\partial x_{i}}$ 
for $i=1,\dots,\ell$. 
Let $D^{(m)}(S):=\bigoplus_{|\bm{\alpha}|=m}S\partial^{\bm{\alpha}}$ 
be the module of differential operators (of order $m$) of $S$, where 
$|\bm{\alpha}|:=\alpha_{1}+\cdots+\alpha_{\ell}$ and 
$\partial^{\bm{\alpha}}:=\partial_{1}^{\alpha_{1}}\cdots
\partial_{\ell}^{\alpha_{\ell}}$ for a multi-index 
$\bm{\alpha}=(\alpha_{1},\dots,\alpha_{\ell})\in\NB^{\ell}$. 
A nonzero element 
$\theta=\sum_{|\bm{\alpha}|=m}f_{\bm{\alpha}}\partial^{\bm{\alpha}}
\in D^{(m)}(S)$ is homogeneous of degree $i$ if 
$f_{\bm{\alpha}}$ is zero or homogeneous of degree $i$ 
for each $\bm{\alpha}$. When $\theta\in D^{(m)}(S)$ 
is homogeneous of degree $i$, we write $\deg(\theta)=i$. 
For a multi-index $\bm{\alpha}$, we put 
\begin{align}
x_{\bm{\alpha}}&:=(x_{1},\dots,x_{1},x_{2},\dots,x_{2},
\dots,x_{\ell},\dots,x_{\ell}),\label{x_a}\\
x_{\bm{\alpha}}^2&:=(x_{1}^2,\dots,x_{1}^2,x_{2}^2,\dots,x_{2}^2,
\dots,x_{\ell}^2,\dots,x_{\ell}^2)\label{x_a^2}
\end{align}
where the number of $x_{i}$ (or $x_{i}^2$) is $\alpha_{i}$. 

Let $\AS$ be a central (hyperplane) arrangement (i.e., every 
hyperplane contains the origin) in $V$. 
Fix a linear form $p_{H}\in V^{\ast}$ such that $\ker(p_{H})=H$ 
for each hyperplane $H\in\AS$, 
and put $Q(\AS):=\prod_{H\in\AS}p_{H}$. 
We call $Q(\AS)$ a defining polynomial of $\AS$. 
We define the {\bf module $D^{(m)}(\AS)$ of 
$\AS$-differential operators} of order $m$ by
$$
D^{(m)}(\AS):=\left\{\theta\in D^{(m)}(S)\mid
\theta(Q(\AS)S)\subseteq Q(\AS)S\right\}.
$$
In the case $m=1$, $D^{(1)}(\AS)$ is the module of 
$\AS$-derivations. We say $\AS$ to be free if $D^{(1)}(\AS)$ 
is a free $S$-module. An excellent reference on 
arrangements is the book by Orlik and Terao \cite{O-T}.

For a commutative $\RB$-algebra $R$, 
let $\DS(R)$ denote the ring of differential operators. 
Then $\DS(S)$ is the Weyl algebra. For an ideal $J$ of $S$, 
let $\DS(J)$ denote the subring of $\DS(S)$ consisting 
of the operators preserving the ideal $J$. 
There is a ring isomorphism:
$$
\DS(S/J)\simeq\DS(J)/J\DS(S)
$$
(see \cite[Theorem 15.5.13]{Mac-Rob}). 
Holm \cite{Holm} showed that $\DS(Q(\AS)S)$ 
decomposes into the direct sum of $D^{(m)}(\AS)$. 
Thus we have an $S$-module isomorphism 
$$
\DS(S/Q(\AS)S)\simeq \frac{\bigoplus_{m\geq 0}
D^{(m)}(\AS)}{Q(\AS)\DS(S)}.
$$

There has been a lot of research on finiteness properties 
of rings of differential operators. 
Systems of generators for $\DS(R)$ are usefull to study 
finiteness properties. 
For example, it is known that $\DS(S/Q(\AS)S)$ is a Noetherian ring 
when $\AS$ is a $2$-dimensional central arrangement, and 
an expression by a basis played a key role in the proof of \cite{Nakashima2}. 
One of the aim to study freeness for 
the module $D^{(m)}(\AS)$ of $\AS$-differential operators 
is to give an $S$-basis (or $S$-generators) of the ring 
of differential operators $\DS(S/Q(\AS)S)$. 
As the first step, we put the study of the module $D^{(2)}(\AS)$ 
into practice when $\AS$ is a Coxeter arrangement. 

Let $W$ be a finite reflection group generated by 
reflections acting on $V$. Naturally $W$ acts on $S$, 
and $W$ acts on the tensor products 
$D^{(m)}(S)\simeq S\otimes_{\RB}\sum_{|\bm{\alpha}|=m}
\RB\partial^{\bm{\alpha}}$. 
The collection of reflection hyperplanes 
of $W$ is called a Coxeter arrangement 
(or a reflection arrangement). 
Coxeter arrangements $\AC_{\ell-1}$, $\BC_{\ell}$ 
and $\DC_{\ell}$ are respectively defined by
\begin{align*}
&\AC_{\ell-1}:=\left\{H_{ij}=\{x_{i}-x_{j}=0\}\mid
1\leq i<j\leq\ell\right\},\\
&\BC_{\ell}:=\left\{H_{i}=\{x_{i}=0\}\mid
i=1,\dots,\ell\right\}\cup
\left\{H_{ij}^{\pm 1}=\{x_{i}\pm x_{j}=0\}\mid
1\leq i<j\leq\ell\right\},\\
&\DC_{\ell}:=\left\{H_{ij}^{\pm 1}
=\{x_{i}\pm x_{j}=0\}\mid 1\leq i<j\leq\ell\right\}.
\end{align*}

From now on, we assume that $\AS$ is a Coxeter arrangement. 
The module of $\AS$-derivations is related to 
the invariant theory of the reflection group corresponding to $\AS$. 
K. Saito has proved that a Coxeter arrangement $\AS$ 
is free, and the module of $\AS$-derivations 
is isomorphic to $S\otimes_{S^W}D^{(1)}(S)^W$ as an 
$S$-module, where $S^W$ and $D^{(1)}(S)^W$ are 
the set of invariant elements of $S$ and $D^{(1)}(S)$, respectively 
 (see, for example, Theorem 6.60 in \cite{O-T}). 
In particular, there exists a basis for 
the module of $\AS$-derivations consisting of invariant elements. 
We can find an explicit basis for $D^{(1)}(\AS)$ in 
\cite{Joz-Sag} when $\AS$ are Coxeter arrangements 
$\AC_{\ell-1}$, $\BC_{\ell}$ and $\DC_{\ell}$. 
However, $D^{(2)}(\AS)$ cannot have bases consisting of 
only invariant elements when $\AS$ are Coxeter arrangements 
$\AC_{\ell-1}$, $\BC_{\ell}$ and $\DC_{\ell}$. 
We prove the assertion above in Section \ref{GA}. 

In this paper we prove that the module 
of differential operators of order $2$ on 
Coxeter arrangements $\AC_{\ell-1}$, $\BC_{\ell}$ 
and $\DC_{\ell}$ are free 
by constructing bases in Section \ref{A,B} and \ref{D}. 
For this purpose, we introduce Cauchy-Sylvester's theorem on 
compound determinants and Saito-Holm's criterion. 
In Section \ref{Cau-Syl-com}, we give some applications of 
the Cauchy-Sylvester's theorem on compound determinants. 

The results of this work (without proofs) have been 
submitted as an extended abstract to FPSAC 2012 
\cite{Nakashima}.
\section{Saito-Holm's criterion}
In this section, we explain Saito-Holm's criterion. 
Put $s_{m}:=\binom{\ell+m-1}{m}$ and 
$t_{m}:=\binom{\ell+m-2}{m-1}$, and set 
$$
\{\bm{\alpha}^{(1)},\dots,\bm{\alpha}^{(s_{m})}\}
=\{\bm{\alpha}\in\NB^{\ell}\mid|\bm{\alpha}|=m\},
$$
where $|\bm{\alpha}|=\alpha_{1}+\cdots+\alpha_{\ell}$ for
a multi-index $\bm{\alpha}\in\NB^{\ell}$. 
For operators $\theta_{1},\dots,\theta_{s_{m}}\in D^{(m)}(\AC)$, 
define the {\bf coefficient matrix 
$M_{m}(\theta_{1},\dots,\theta_{s_{m}})$} of the operators 
$\theta_{1},\dots,\theta_{s_{m}}$ as follows:
$$
M_{m}(\theta_{1},\dots,\theta_{s_{m}}):=
\left[\theta_{i}\left(\frac{x^{\bm{\alpha}^{(j)}}}{\bm{\alpha}^{(j)}!}
\right)\right]_{1\leq i,j\leq s_{m}},
$$
where $\bm{\alpha}!=\alpha_{1}!\cdots\alpha_{\ell}!$. 
Thus the $(i,j)$-entry of the coefficient matrix is 
the polynomial coefficient of $\partial^{\bm{\alpha}^{(j)}}$ 
in $\theta_{i}$. 

The following criterion was originally given by Saito \cite{Sai} 
in the case $m=1$, and was generalized by Holm \cite{Holm-phd} 
into the case $m$ general.
\begin{Prop}[Saito-Holm's criterion]\label{Saito-Holm}
Let
$\theta_{1},\dots,\theta_{s_{m}}\in D^{(m)}(\AS)$ 
be homogeneous operators. Then 
the following two conditions are equivalent:
\begin{enumerate}
\item[$(1)$] $\det M_{m}(\theta_{1},\dots,\theta_{s_{m}})
=cQ^{t_{m}}$ for some $c\in \RB^{\times}$.
\item[$(2)$] $\theta_{1},\dots,\theta_{s_{m}}$ form a basis
 for $D^{(m)}(\AS)$ over $S$.
\end{enumerate}
\end{Prop}

When $D^{(m)}(\AS)$ is a free $S$-module, 
we define the {\bf exponents} of $D^{(m)}(\AS)$ 
to be the multi-set of degrees of a homogeneous basis 
$\{\theta_{1},\dots,\theta_{s_{m}}\}$ for $D^{(m)}(\AS)$, 
which is denoted by $\exp D^{(m)}(\AS)$:
$$
\exp D^{(m)}(\AS)=
\left\{\deg(\theta_{1}),\dots,\deg(\theta_{s_{m}})\right\}.
$$
\section{Cauchy-Sylvester's theorem on compound determinants}
\label{Cau-Syl-com}
Throughout this paper, we assume $\ell\geq m$. 
In this section, we will follow the notation of 
the paper by Ito and Okada \cite{I-O} as far as possible. 
We denote by $\succ$ the lexicographic order on $\ZB^m$. 
That is, for $\mu=\left(\mu_{1},\dots,\mu_{m}\right)$ 
and $\nu=\left(\nu_{1},\dots,\nu_{m}\right)\in\ZB^m$, 
we write $\mu\succ\nu$ 
if there exist an index $k$ such that
\begin{align*}
\mu_{1}=\nu_{1},\dots,\mu_{k-1}=\nu_{k-1},\ 
{\rm and}\ \mu_{k}>\nu_{k}.
\end{align*}
Set 
\begin{align*}
Z:=\left\{
\mu=\left(\mu_{1},\dots,\mu_{m}\right)\in\ZB^m
\mid 1\leq\mu_{1}<\mu_{2}<\cdots<\mu_{m}\leq \ell
\right\}.
\end{align*}\
Then $Z$ is a totally ordered subset of $\ZB^m$. 
Put $x_{\mu}:=(x_{\mu_{1}},\dots,x_{\mu_{m}})\in S^{m}$.

Let $A=\left(a_{i,j}\right)_{1\leq i,j\leq \ell}$ 
be a square matrix of order $\ell$. For $\mu,\nu\in Z$ put 
$$
A_{\mu,\nu}:=\left(a_{\mu_{i},\nu_{j}}\right)_{1\leq i,j\leq m}.
$$
We define the $m$-th {\bf compound matrix} $A^{(m)}$ by
$$
A^{(m)}:=\left(\det A_{\mu,\nu}\right)_{\mu,\nu\in Z},
$$
where the rows and columns are arranged 
in the increasing order on $Z$. 

The following was obtained by Cauchy and Sylvester 
(see, for example, \cite[Proposition 3.1]{I-O}).
\begin{Prop}[Cauchy-Sylvester]
Let $A=\left(a_{i,j}\right)_{1\leq i,j\leq \ell}$ be a 
square matrix. Then the determinant of the $m$-th 
compound matrix $A^{(m)}$ is given by 
\begin{align}
\det A^{(m)}=\left(\det A\right)^{\binom{\ell-1}{m-1}}.
\label{ca-sy}
\end{align}
\end{Prop}
Put
$$
\Lambda:=\left\{\lambda=
\left(\lambda_{1},\dots,\lambda_{m}\right)\in\ZB^m\mid
\ell-m\geq\lambda_{1}\geq\lambda_{2}\geq\cdots\geq
\lambda_{m}\geq0\right\}.
$$
We regard $\Lambda$ as a totally ordered subset of $\ZB^m$ 
by the order $\succ$. Then the map 
$$
Z\ni\left(\mu_{1},\dots,\mu_{m}\right)
\longmapsto\left(\ell-m+1-\mu_{1},
\ell-m+2-\mu_{2},\dots,\ell-\mu_{m}\right)\in\Lambda
$$
is a bijection between $\Lambda$ and $Z$, and this bijection 
reverses the ordering on $\Lambda$ and $Z$. 

For $\lambda\in\Lambda$, 
we define the following symmetric polynomials and 
a Laurent polynomial:
\begin{align}
s^{\AC}_{\lambda}&:=
\frac{\det(t_{i}^{\lambda_{j}+m-j})_{1\leq i,j\leq m}}
{\det(t_{i}^{m-j})_{1\leq i,j\leq m}}
\in S[t_{1},\dots,t_{m}],\label{s^A}\\
s^{\BC}_{\lambda}&:=
\frac{\det(t_{i}^{2(\lambda_{j}+m-j)+1})_{1\leq i,j\leq m}}
{\det(t_{i}^{2(m-j)})_{1\leq i,j\leq m}}
\in S[t_{1},\dots,t_{m}],\label{s^B}\\
s^{\DC}_{\lambda}&:=
\frac{\det(t_{i}^{2(\lambda_{j}+m-j)-1})_{1\leq i,j\leq m}}
{\det(t_{i}^{2(m-j)})_{1\leq i,j\leq m}}
\in S[t_{1}^{\pm 1},\dots,t_{m}^{\pm 1}].\label{s^D}
\end{align}
The polynomial $s^{\AC}_{\lambda}$ is the Schur polynomial 
corresponding to the partition $\lambda$. 
The Laurent polynomials $s^{\BC}_{\lambda}$ and $s^{\DC}_{\lambda}$ 
may be expressed by $s^{\AC}_{\lambda}$ as follows:
\begin{align}
s^{\BC}_{\lambda}&=t_{1}\cdots t_{m}\cdot
\frac{\det((t_{i}^2)^{\lambda_{j}+m-j})_{1\leq i,j\leq m}}
{\det((t_{i}^2)^{m-j})_{1\leq i,j\leq m}}
=t_{1}\cdots t_{m}s^{\AC}_{\lambda}(t_{1}^2,\dots,t_{m}^2)
\label{exprB}\\
s^{\DC}_{\lambda}&=\frac{1}{t_{1}\cdots t_{m}}\cdot
\frac{\det((t_{i}^2)^{\lambda_{j}+m-j})_{1\leq i,j\leq m}}
{\det((t_{i}^2)^{m-j})_{1\leq i,j\leq m}}
=\frac{1}{t_{1}\cdots t_{m}}s^{\AC}_{\lambda}(t_{1}^2,\dots,t_{m}^2)
\label{exprD}
\end{align}
We remark that $s^{\DC}_{\lambda}$ is a symmetric 
polynomial if $\lambda_{m}\geq 1$. Now 
the degrees of these Laurent polynomials are following:
\begin{align}
\deg s^{\AC}_{\lambda}=|\lambda|,\quad 
\deg s^{\BC}_{\lambda}=2|\lambda|+m, \quad
\deg s^{\DC}_{\lambda}=2|\lambda|-m,\label{degs}
\end{align}
where $|\lambda|:=\lambda_{1}+\cdots+\lambda_{m}$.
\begin{Prop}\label{det}
We have the following determinant identities:
\begin{align}
\det\left(s^{\AC}_{\lambda}(x_{\mu})\right)
_{\substack{\lambda\in\Lambda\\ \mu\in Z}}
&=\left[\prod_{1\leq i<j\leq\ell}(x_{i}-x_{j})
\right]^{\binom{\ell-2}{m-1}},\label{detsa}\\
\det\left(s^{\BC}_{\lambda}(x_{\mu})\right)
_{\substack{\lambda\in\Lambda\\ \mu\in Z}}
&=\left(x_{1}\cdots x_{\ell}\right)^{\binom{\ell-1}{m-1}}
\left[\prod_{1\leq i<j\leq\ell}(x_{i}^2-x_{j}^2)
\right]^{\binom{\ell-2}{m-1}},\label{detsb}\\
\det\left(s^{\DC}_{\lambda}(x_{\mu})\right)
_{\substack{\lambda\in\Lambda\\ \mu\in Z}}
&=\frac{1}{(x_{1}\cdots x_{\ell})^{\binom{\ell-1}{m-1}}}
\left[\prod_{1\leq i<j\leq\ell}(x_{i}^2-x_{j}^2)
\right]^{\binom{\ell-2}{m-1}}.\label{detsd}
\end{align}
\end{Prop}
\begin{proof}
Apply the formula (\ref{ca-sy}) to the matrices 
$A=(x_{i}^{\ell-j})_{1\leq i,j\leq\ell}$, 
$A=(x_{i}^{2(\ell-j)+1})_{1\leq i,j\leq\ell}$ 
and $A=(x_{i}^{2(\ell-j)-1})_{1\leq i,j\leq\ell}$.
\end{proof}
We will use these determinant identities for proving that 
$D^{(2)}(\AS)$ are free when $\AS$ are 
reflection arrangements types A, B and D 
in Section \ref{A,B} and \ref{D}.
\begin{Examp}
We assume that $\ell=3, m=2$. Then 
$$
\Lambda=\{(\lambda_{1},\lambda_{2})\mid
1\geq \lambda_{1}\geq \lambda_{2}\geq 0\}=
\{(1,1),(1,0),(0,0)\}.
$$
The Schur polynomials are following:
\begin{align*}
s^{\AC}_{(1,1)}(t_{1},t_{2})&=
\frac{t_{1}^2 t_{2}-t_{1}t_{2}^2}{t_{1}-t_{2}}=t_{1}t_{2},\\
s^{\AC}_{(1,0)}(t_{1},t_{2})&=
\frac{t_{1}^2-t_{2}^2}{t_{1}-t_{2}}=t_{1}+t_{2},\\
s^{\AC}_{(0,0)}(t_{1},t_{2})&=
\frac{t_{1}-t_{2}}{t_{1}-t_{2}}=1.
\end{align*}

Let 
\begin{align*}
A=\left[
\begin{matrix}
x_{1}^2&x_{1}&1\\
x_{2}^2&x_{2}&1\\
x_{3}^2&x_{3}&1
\end{matrix}
\right]
.
\end{align*}
Then the determinant of $A$ 
is Vandermonde's determinant, and is equal to 
the direct product $(x_{1}-x_{2})(x_{1}-x_{3})(x_{2}-x_{3})$. 

Let us consider the second compound matrix $A^{(2)}$: 
\begin{align*}
A^{(2)}&=\left[
\begin{matrix}
x_{1}^2 x_{2}-x_{1} x_{2}^2&x_{1}^2-x_{2}^2&x_{1}-x_{2}\\
x_{1}^2 x_{3}-x_{1} x_{3}^2&x_{1}^2-x_{3}^2&x_{1}-x_{3}\\
x_{2}^2 x_{3}-x_{2} x_{3}^2&x_{2}^2-x_{3}^2&x_{2}-x_{3}
\end{matrix}
\right]
\\
&=\left[
\begin{matrix}
(x_{1}-x_{2})s^{\AC}_{(1,1)}(x_{1},x_{2})&
(x_{1}-x_{2})s^{\AC}_{(1,0)}(x_{1},x_{2})&
(x_{1}-x_{2})s^{\AC}_{(0,0)}(x_{1},x_{2})\\
(x_{1}-x_{3})s^{\AC}_{(1,1)}(x_{1},x_{3})&
(x_{1}-x_{3})s^{\AC}_{(1,0)}(x_{1},x_{3})&
(x_{1}-x_{3})s^{\AC}_{(0,0)}(x_{1},x_{3})\\
(x_{2}-x_{3})s^{\AC}_{(1,1)}(x_{2},x_{3})&
(x_{2}-x_{3})s^{\AC}_{(1,0)}(x_{2},x_{3})&
(x_{2}-x_{3})s^{\AC}_{(0,0)}(x_{2},x_{3})
\end{matrix}
\right].
\end{align*}
By the identity (\ref{ca-sy}), we have the determinant identity 
\begin{align*}
\begin{vmatrix}
s^{\AC}_{(1,1)}(x_{1},x_{2})&
s^{\AC}_{(1,0)}(x_{1},x_{2})&
s^{\AC}_{(0,0)}(x_{1},x_{2})\\
s^{\AC}_{(1,1)}(x_{1},x_{3})&
s^{\AC}_{(1,0)}(x_{1},x_{3})&
s^{\AC}_{(0,0)}(x_{1},x_{3})\\
s^{\AC}_{(1,1)}(x_{2},x_{3})&
s^{\AC}_{(1,0)}(x_{2},x_{3})&
s^{\AC}_{(0,0)}(x_{2},x_{3})
\end{vmatrix}
=(x_{1}-x_{2})(x_{1}-x_{3})(x_{2}-x_{3}).
\end{align*}
\end{Examp}
\begin{Examp}
Let $\ell=3, m=2$. Then 
\begin{align*}
s^{\BC}_{(1,1)}(t_{1},t_{2})=
\frac{t_{1}^5 t_{2}^3-t_{1}^3 t_{2}^5}{t_{1}^2-t_{2}^2},\ 
s^{\BC}_{(1,0)}(t_{1},t_{2})=
\frac{t_{1}^5 t_{2}-t_{1}t_{2}^5}{t_{1}^2-t_{2}^2},\ 
s^{\BC}_{(0,0)}(t_{1},t_{2})=
\frac{t_{1}^3 t_{2}-t_{1}t_{2}^3}{t_{1}^2-t_{2}^2}.
\end{align*}

Let 
\begin{align*}
B=\left[
\begin{matrix}
x_{1}^5&x_{1}^3&x_{1}\\
x_{2}^5&x_{2}^3&x_{2}\\
x_{3}^5&x_{3}^3&x_{3}
\end{matrix}
\right].
\end{align*}
Then the determinant $\det B$ is equal to  
$x_{1}x_{2}x_{3}(x_{1}^2-x_{2}^2)(x_{1}^2-x_{3}^2)(x_{2}^2-x_{3}^2)$. 
Since $\det B^{(2)}=\left( x_{1}x_{2}x_{3}(x_{1}^2-x_{2}^2)
(x_{1}^2-x_{3}^2)(x_{2}^2-x_{3}^2)\right)^2$, we have 
\begin{align*}
\begin{vmatrix}
s^{\BC}_{(1,1)}(x_{1},x_{2})&
s^{\BC}_{(1,0)}(x_{1},x_{2})&
s^{\BC}_{(0,0)}(x_{1},x_{2})\\
s^{\BC}_{(1,1)}(x_{1},x_{3})&
s^{\BC}_{(1,0)}(x_{1},x_{3})&
s^{\BC}_{(0,0)}(x_{1},x_{3})\\
s^{\BC}_{(1,1)}(x_{2},x_{3})&
s^{\BC}_{(1,0)}(x_{2},x_{3})&
s^{\BC}_{(0,0)}(x_{2},x_{3})
\end{vmatrix}
=x_{1}^2 x_{2}^2 x_{3}^2(x_{1}^2-x_{2}^2)(x_{1}^2-x_{3}^2)(x_{2}^2-x_{3}^2).
\end{align*}
\end{Examp}

\section{Type $A$ and $B$}\label{A,B}
Let $\AS$ be an arbitrary arrangement. 
By \cite[Proposition 2.3]{Holm} and \cite[Theorem 2.4]{Holm}, 
we have
\begin{align}
D^{(m)}(\AS)=\bigcap_{H\in\AS}D^{(m)}(p_{H}S),\label{inter}
\end{align}
where $D^{(m)}(p_{H}S)=\left\{
\theta\in D^{(m)}(S)\mid\theta(p_{H}x^{\bm{\alpha}})\in p_{H}S\ 
{\rm for\ any}\ |\bm{\alpha}|=m-1\right\}$ 
for $H\in\AS$. 

Recall that the defining polynomials of Coxeter arrangements 
$\AC_{\ell-1}$ and $\BC_{\ell}$ of types $A$ and $B$ are 
\begin{align*}
&Q(\AC_{\ell-1})=\prod_{1\leq i<j\leq\ell}(x_{i}-x_{j}),\\
&Q(\BC_{\ell})=x_{1}\cdots x_{\ell}\prod_{1\leq i<j\leq\ell}(x_{i}^2-x_{j}^2).
\end{align*}
We introduce some operators which are in $D^{(m)}(\AC_{\ell-1})$ or 
$D^{(m)}(\BC_{\ell})$. 
By using these operators, we construct bases for the modules 
$D^{(2)}(\AC_{\ell-1})$ and $D^{(2)}(\BC_{\ell})$ 
of differential operators of order $2$ on $\AC_{\ell-1}$ and $\BC_{\ell}$.

Let $k=1,\dots,\ell$, and put 
$h^{\AC}_{k}:=(x_{k}-x_{1})\cdots(x_{k}-x_{k-1})
(x_{k}-x_{k+1})\cdots(x_{k}-x_{\ell})$ and 
$h^{\BC}_{k}:=x_{k}(x_{k}^2-x_{1}^2)\cdots(x_{k}^2-x_{k-1}^2)
(x_{k}^2-x_{k+1}^2)\cdots(x_{k}^2-x_{\ell}^2)$. 
We define operators $\eta^{\AC}_{k}$ and $\eta^{\BC}_{k}$ in 
$D^{(m)}(S)$ as follows:
\begin{align*}
\eta^{\AC}_{k}:=h^{\AC}_{k}\frac{1}{m!}\partial_{k}^m,
\quad\eta^{\BC}_{k}:=h^{\BC}_{k}\frac{1}{m!}\partial_{k}^m.
\end{align*}
Then $\deg\eta^{\AC}_{k}=\ell-1$ and $\deg\eta^{\BC}_{k}=2\ell-1$. 

It is convenient to write $f\doteq g$ for $f,g\in S$ if 
$f=cg$ for some $c\in \RB^{\times}$. 
\begin{Prop}\label{ope-ab1}
For $k=1,\dots,\ell$, we have that 
$\eta^{\AC}_{k}\in D^{(m)}(\AC_{\ell-1})$ and 
$\eta^{\BC}_{k}\in D^{(m)}(\BC_{\ell})$.
\end{Prop}
\begin{proof}
For any $1\leq i<j\leq \ell$ and a multi-index $\bm{\beta}$ 
with $|\bm{\beta}|=m-1$, 
\begin{align*}
\frac{1}{m!}\partial_{k}^m
\left((x_{i}\pm x_{j})x^{\bm{\beta}}\right)=
\begin{cases}
1&{\rm if}\ i=k\ {\rm and}\ \beta_{i}+1=m,\\
\pm 1&{\rm if}\ j=k\ {\rm and}\ \beta_{j}+1=m,\\
0&{\rm otherwise}.
\end{cases}
\end{align*}
If $i=k$ and $\beta_{i}+1=m$ or $j=k$ and $\beta_{i}+1=m$, then 
$\eta^{\AC}_{k}((x_{i}-x_{j})\cdot x^{\bm{\beta}})
\doteq h^{\AC}_{k}\in (x_{i}-x_{j})S$. 
Therefore we obtain $\eta^{\AC}_{k}\in D^{(m)}(\AC_{\ell-1})$
from (\ref{inter}). 

Similarly we have 
$\eta^{\BC}_{k}\in \bigcap_{1\leq i<j\leq \ell}
D^{(m)}\left((x_{i}^2-x_{j}^2)S\right)$. 
For $i=1,\dots,\ell$ and a multi-index $\bm{\beta}$ 
with $|\bm{\beta}|=m-1$, we have 
\begin{align*}
\eta^{\BC}_{k}
\left(x_{i}\cdot x^{\bm{\beta}}\right)=
\begin{cases}
h^{\BC}_{k}&{\rm if}\ i=k\ {\rm and}\ \beta_{i}+1=m,\\
0&{\rm otherwise}.
\end{cases}
\end{align*}
This leads to that $\eta^{\BC}_{k}\in
\bigcap_{i=1}^{\ell} D^{(m)}\left(x_{i}S\right)$. 
Therefore we obtain $\eta^{\BC}_{k}\in D^{(m)}(\BC_{\ell})$.
\end{proof}
For a Laurent polynomial 
$f(t_{1},\dots,t_{m})\in S[t_{1}^{\pm 1},\dots,t_{m}^{\pm 1}]$ 
satisfying $f(x_{\bm{\alpha}})\in S$ for any $\bm{\alpha}$ 
with $|\bm{\alpha}|=m$, we define an operator 
$$
\theta_{f}:=\sum_{|\bm{\alpha}|=m}
f\left(x_{\bm{\alpha}}\right)\frac{1}{\bm{\alpha}!}
\partial^{\bm{\alpha}}.
$$
We say a Laurent polynomial 
$f(t_{1},\dots,t_{m})$ is symmetric if 
$$
f(t_{1},\dots,t_{i},\dots,t_{j},\dots,t_{m})=
f(t_{1},\dots,t_{j},\dots,t_{i},\dots,t_{m})
$$
for all pairs $(i,j)$.
\begin{Lem}\label{symm}
Assume that $f(t_{1},\dots,t_{m})$ is a symmetric 
Laurent polynomial. 
Then we have that $\theta_{f}\in D^{(m)}(\AC_{\ell-1})$.
\end{Lem}
\begin{proof}
Since $f(t_{1},\dots,t_{m})$ is symmetric, 
we have 
\begin{align*}
\theta_{f}\left((x_{i}-x_{j})\cdot x^{\bm{\beta}}\right)|
_{x_{i}=x_{j}}=\left(f(x_{\bm{\beta}+\bm{e}_{i}})-
f(x_{\bm{\beta}+\bm{e}_{j}})\right)|_{x_{i}=x_{j}}=0
\end{align*}
for any $1\leq i<j\leq \ell$ and a multi-index $\bm{\beta}$ 
with $|\bm{\beta}|=m-1$. We obtain 
$\theta_{f}\left((x_{i}-x_{j})\cdot x^{\bm{\beta}}\right)\in 
(x_{i}-x_{j})S$. Hence it follows from (\ref{inter}) that 
$\theta_{f}\in D^{(m)}(\AC_{\ell-1})$.
\end{proof}
For $\lambda\in\Lambda$, define operators
\begin{align*}
\theta^{\AC}_{\lambda}:=\sum_{|\bm{\alpha}|=m}
s^{\AC}_{\lambda}\left(x_{\bm{\alpha}}\right)\frac{1}{\bm{\alpha}!}
\partial^{\bm{\alpha}},\quad
\theta^{\BC}_{\lambda}:=\sum_{|\bm{\alpha}|=m}
s^{\BC}_{\lambda}\left(x_{\bm{\alpha}}\right)\frac{1}{\bm{\alpha}!}
\partial^{\bm{\alpha}}. 
\end{align*}
Then 
$\deg \theta^{\AC}_{\lambda}=|\lambda|,
\deg \theta^{\BC}_{\lambda}=2|\lambda|+m$ 
by the formula (\ref{degs}).
\begin{Prop}\label{ope-ab2}
For $\lambda\in\Lambda$, we have 
$\theta^{\AC}_{\lambda}\in D^{(m)}(\AC_{\ell-1})$ and 
$\theta^{\BC}_{\lambda}\in D^{(m)}(\BC_{\ell})$.
\end{Prop}
\begin{proof}
Since Laurent polynomials $s^{\AC}_{\lambda}$ and 
$s^{\BC}_{\lambda}$ are symmetric, we obtain 
$\theta^{\AC}_{\lambda},\theta^{\BC}_{\lambda}\in
 D^{(m)}(\AC_{\ell-1})$ by Lemma \ref{symm}.

By (\ref{inter}), we can write
$$
D^{(m)}(\BC_{\ell})=D^{(m)}(\AC_{\ell-1})\cap
\left(\bigcap_{i=1}^{\ell}D^{(m)}(x_{i}S)\right)
\cap\left(\bigcap_{1\leq i<j\leq\ell}
D^{(m)}\left((x_{i}+x_{j})S\right)\right).
$$
Thus we only need to prove that 
$$
\theta^{\BC}_{\lambda}\in
\left(\bigcap_{i=1}^{\ell}D^{(m)}(x_{i}S)\right)
\quad{\rm and}\quad
\theta^{\BC}_{\lambda}\in\left(\bigcap_{1\leq i<j\leq\ell}
D^{(m)}\left((x_{i}+x_{j})S\right)\right). 
$$

For any $i=1,\dots,\ell$ and a multi-index $\bm{\beta}$ 
with $|\bm{\beta}|=m-1$, we have 
\begin{align*}
\theta^{\BC}_{\lambda}(x_{i}\cdot x^{\bm{\beta}})=
s^{\BC}_{\lambda}(x_{\bm{\beta}+\bm{e}_{i}})
=x_{i}\cdot x^{\bm{\beta}}
s^{\AC}_{\lambda}(x_{\bm{\beta}+\bm{e}_{i}}^2)
\in x_{i}S.
\end{align*}
This implies 
$\bigcap_{i=1}^{\ell}\theta^{\BC}_{\lambda}\in D^{(m)}(x_{i}S)$. 
 
For any $1\leq i<j\leq \ell$ and a multi-index $\bm{\beta}$ 
with $|\bm{\beta}|=m-1$, 
\begin{align*}
\theta^{\BC}_{\lambda}((x_{i}+x_{j})\cdot x^{\bm{\beta}})=
s^{\BC}_{\lambda}(x_{\bm{\beta}+\bm{e}_{i}})
+s^{\BC}_{\lambda}(x_{\bm{\beta}+\bm{e}_{j}})
=x^{\bm{\beta}}\left(
x_{i}s^{\AC}_{\lambda}(x_{\bm{\beta}+\bm{e}_{i}}^2)
+x_{j}s^{\AC}_{\lambda}(x_{\bm{\beta}+\bm{e}_{j}}^2)
\right)
\end{align*}
Then we have 
$\theta^{\BC}_{\lambda}((x_{i}+x_{j})\cdot x^{\bm{\beta}})
|_{x_{i}=-x_{j}}=0$, and this implies 
$\theta^{\BC}_{\lambda}\left((x_{i}+x_{j})\cdot 
x^{\bm{\beta}}\right)\in (x_{i}+x_{j})S$.
Hence we obtain 
$\theta^{\BC}_{\lambda}\in D^{(m)}(\BC_{\ell})$.
\end{proof}
\begin{Them}\label{basisAB}
Let $m=2$.
\begin{enumerate}
\item[$(1)$]The set 
\begin{align*}
C_{\AC}&:=\left\{\eta^{\AC}_{i}\mid i=1,\dots\ell\right\}\cup
\left\{\theta^{\AC}_{\lambda}\mid\lambda\in\Lambda\right\}
\end{align*}
forms an $S$-basis for $D^{(2)}(\AC_{\ell-1})$. Hence 
$$
\exp D^{(2)}(\AC_{\ell-1})=\{\ell-1,\dots,\ell-1\}
\cup\{|\lambda|\mid\lambda\in\Lambda\}.
$$
\item[$(2)$]The set 
\begin{align*}
C_{\BC}&:=\left\{\eta^{\BC}_{i}\mid i=1,\dots\ell\right\}\cup
\left\{\theta^{\BC}_{\lambda}\mid\lambda\in\Lambda\right\}
\end{align*}
forms an $S$-basis for $D^{(2)}(\BC_{\ell})$. Hence 
$$
\exp D^{(2)}(\BC_{\ell})=\{2\ell-1,\dots,2\ell-1\}
\cup\{2|\lambda|+2\mid\lambda\in\Lambda\}.
$$
\end{enumerate}
\end{Them}
\begin{proof}
$(1)$ All operators in $C_{\AC}$ belong to $D^{(2)}(\AC_{\ell-1})$ 
by Proposition \ref{ope-ab1} and Proposition \ref{ope-ab2}. 

By Proposition \ref{Saito-Holm}, we only need to prove that 
the determinant of the coefficient matrix 
$M_{m}(C_{\AC})$ of the operators of $C_{\AC}$ is equal to 
$Q(\AC_{\ell-1})^{\ell}$ up to a nonzero constant. 
By Proposition \ref{det}, we obtain 
$\det\left(s^{\AC}_{\lambda}(x_{\bm{\alpha}})\right)
_{\lambda\in\Lambda,\bm{\alpha}\in Z}=Q(\AC)^{\ell-2}$. 
Hence we have 
\begin{align*}
\det M_{m}(C_{\AC})\doteq Q(\AC_{\ell-1})^{2}
\begin{vmatrix}
I_{\ell}&\ast\\
0&\det\left(s^{\AC}_{\lambda}(x_{\bm{\alpha}})\right)
_{\substack{\lambda\in\Lambda\\ \bm{\alpha}\in Z}}
\end{vmatrix}
=Q(\AC_{\ell-1})^{2}\cdot Q(\AC_{\ell-1})^{\ell-2}=Q(\AC_{\ell-1})^{\ell}.
\end{align*}

$(2)$ We have an identity 
\begin{align*}
\det M_{m}(C_{\BC})\doteq 
x_{1}\cdots x_{\ell}\left(\prod_{1\leq i<j\leq\ell}
(x_{i}^2-x_{j}^2)\right)^2
\begin{vmatrix}
I_{\ell}&\ast\\
0&\det\left(s^{\BC}_{\lambda}(x_{\bm{\alpha}})\right)
_{\substack{\lambda\in\Lambda\\ \bm{\alpha}\in Z}}
\end{vmatrix}
=Q(\BC_{\ell})^{\ell}
\end{align*}
by Proposition \ref{det}. Then the rest of proof 
for $(2)$ is similar to the one for $(1)$.
\end{proof}
\begin{Examp}
Let $\ell=3, m=2$. Then we have 
$s_{2}=\binom{3+2-1}{2}=6$ and $t_{2}=\binom{3+2-2}{2-1}=3$. 
List the operators of the set $C_{\AC}$:
\begin{align*}
&\eta^{\AC}_{1}=(x_{1}-x_{2})(x_{1}-x_{3})\frac{1}{2}\partial_{1}^2,\\
&\eta^{\AC}_{2}=(x_{2}-x_{1})(x_{2}-x_{3})\frac{1}{2}\partial_{2}^2,\\
&\eta^{\AC}_{3}=(x_{3}-x_{1})(x_{3}-x_{2})\frac{1}{2}\partial_{3}^2,\\
&\theta^{\AC}_{(1,1)}=x_{1}^2\frac{1}{2}\partial_{1}^2+
x_{2}^2\frac{1}{2}\partial_{2}^2+x_{3}^2\frac{1}{2}\partial_{3}^2+
x_{1}x_{2}\partial_{1}\partial_{2}+x_{1}x_{3}\partial_{1}\partial_{3}+
x_{2}x_{3}\partial_{2}\partial_{3},\\
&\theta^{\AC}_{(1,0)}=2x_{1}\frac{1}{2}\partial_{1}^2+
2x_{2}\frac{1}{2}\partial_{2}^2+2x_{3}\frac{1}{2}\partial_{3}^2+
(x_{1}+x_{2})\partial_{1}\partial_{2}+(x_{1}+x_{3})\partial_{1}\partial_{3}+
(x_{2}+x_{3})\partial_{2}\partial_{3},\\
&\theta^{\AC}_{(0,0)}=\frac{1}{2}\partial_{1}^2+
\frac{1}{2}\partial_{2}^2+\frac{1}{2}\partial_{3}^2+
\partial_{1}\partial_{2}+\partial_{1}\partial_{3}+
\partial_{2}\partial_{3}.
\end{align*}
Hence the determinant of the coefficient matrix of operators above is 
\begin{align*}
&\det M_{2}\left(\eta^{\AC}_{1},\eta^{\AC}_{2},\eta^{\AC}_{3},
\theta^{\AC}_{(1,1)},\theta^{\AC}_{(1,0)},
\theta^{\AC}_{(0,0)}\right)\\
=&
\begin{vmatrix}
(x_{1}-x_{2})(x_{1}-x_{3})&0&0&\frac{1}{2}x_{1}^2&x_{1}&\frac{1}{2}\\
0&(x_{2}-x_{1})(x_{2}-x_{3})&0&\frac{1}{2}x_{2}^2&x_{2}&\frac{1}{2}\\
0&0&(x_{3}-x_{1})(x_{3}-x_{2})&\frac{1}{2}x_{3}^2&x_{2}&\frac{1}{2}\\
0&0&0&x_{1}x_{2}&x_{1}+x_{2}&1\\
0&0&0&x_{1}x_{3}&x_{1}+x_{3}&1\\
0&0&0&x_{2}x_{3}&x_{2}+x_{3}&1
\end{vmatrix}
\\
=&-(x_{1}-x_{2})^2(x_{1}-x_{3})^2(x_{2}-x_{3})^2
\begin{vmatrix}
x_{1}x_{2}&x_{1}+x_{2}&1\\
x_{1}x_{3}&x_{1}+x_{3}&1\\
x_{2}x_{3}&x_{2}+x_{3}&1
\end{vmatrix}
\\
\doteq& Q(\AC_{2})^{3}.
\end{align*}
\end{Examp}
\section{Type $D$}\label{D}
In this section, we assume $m=2$, and we 
construct a basis for $D^{(2)}(\DC_{\ell})$. 
Recall the defining polynomial 
$Q(\DC_{\ell})=\prod_{1\leq i<j\leq\ell}(x_{i}^2-x_{j}^2)$
of the Coxeter arrangement of type $D$. 

Set 
\begin{align*}
\Lambda^{'}&:=\left\{\lambda=(\lambda_{1},
\lambda_{2})\mid\ell-2\geq\lambda_{1}\geq
\lambda_{2}\geq1\right\},\\
\Lambda^{''}&:=\left\{\lambda=(\lambda_{1},
\lambda_{2})\mid\ell-2\geq\lambda_{1}\geq 0,
\lambda_{2}=0\right\}.
\end{align*}
Then $\Lambda=\Lambda^{'}\cup\Lambda^{''}$. 
Put $\lambda^{(0)}:=(0,0)$. 
We define operators $\theta^{\DC}_{\lambda}$ as follows:
\begin{align}
\theta^{\DC}_{\lambda}&:=\sum_{|\bm{\alpha}|=2}
s^{\DC}_{\lambda}\left(x_{\bm{\alpha}}\right)\frac{1}{\bm{\alpha}!}
\partial^{\bm{\alpha}}
\quad{\rm if}\ \lambda\in\Lambda^{'},\label{thetaD1}\\
\theta^{\DC}_{\lambda}&:=(x_{1}\cdots x_{\ell})
\sum_{|\bm{\alpha}|=2}s^{\DC}_{\lambda}\left(x_{\bm{\alpha}}\right)
\frac{1}{\bm{\alpha}!}\partial^{\bm{\alpha}}\quad
{\rm if}\ \lambda\in\Lambda^{''}\setminus\{\lambda^{(0)}\},\label{thetaD2}\\
\theta^{\DC}_{\lambda}&:=(x_{1}\cdots x_{\ell})^2
\sum_{|\bm{\alpha}|=2}s^{\DC}_{\lambda}\left(x_{\bm{\alpha}}\right)
\frac{1}{\bm{\alpha}!}\partial^{\bm{\alpha}}
\quad{\rm if}\ \lambda=\lambda^{(0)}.\label{thetaD3}
\end{align}
If $\lambda\in\Lambda^{'}$, then we have 
\begin{align*}
s^{\DC}_{\lambda}=
\frac{\det(t_{i}^{2(\lambda_{j}-1+2-j)+1})_{1\leq i,j\leq 2}}
{\det(t_{i}^{2(2-j)})_{1\leq i,j\leq 2}}
=s^{\BC}_{\lambda-\bm{1}},
\end{align*}
where $\lambda-\bm{1}=(\lambda_{1}-1,\lambda_{2}-1)$. 

If $\lambda\in\Lambda^{''}\setminus\{\lambda^{(0)}\}$, then 
$$
s^{\DC}_{\lambda}=\frac{t_{1}^{2\lambda_{1}+1}\cdot t_{2}^{-1}
-t_{2}^{2\lambda_{1}+1}\cdot t_{1}^{-1}}{t_{1}^2-t_{2}^2}=
\frac{1}{t_{1}t_{2}}
\sum_{j=0}^{\lambda_{1}}t_{1}^{2j} t_{2}^{2(\lambda_{1}-j)}.
$$
Thus 
$(x_{1}\cdots x_{\ell})s^{\DC}_{\lambda}\left(x_{\bm{\alpha}}\right)$ is 
a polynomial for any multi-index $\bm{\alpha}$ 
with $|\bm{\alpha}|=2$. 

We have 
$$
\theta^{\DC}_{\lambda^{(0)}}=(x_{1}\cdots x_{\ell})^2
\left(\sum_{i=1}^{\ell}\frac{1}{2x_{i}^2}\partial_{i}^2
+\sum_{1\leq i<j\leq \ell}\frac{1}{x_{i}x_{j}}
\partial_{i}\partial_{j}\right).
$$

Hence $\theta^{\DC}_{\lambda}$ for any $\lambda\in\Lambda$. 
The degrees of these operators are as follows:
\begin{align*}
\deg\theta^{\DC}_{\lambda}&=2|\lambda|-2
=2\lambda_{1}+2\lambda_{2}-2
\quad{\rm if}\quad\lambda\in\Lambda^{'},\\
\deg\theta^{\DC}_{\lambda}&=2\lambda_{1}-2+\ell
\quad{\rm if}\quad\lambda\in\Lambda^{''}\setminus
\{\lambda^{(0)}\},\\
\deg\theta^{\DC}_{\lambda}&=2\ell-2
\quad{\rm if}\quad\lambda=\lambda^{(0)}.
\end{align*}
\begin{Prop}\label{ope-d1}
For $\lambda\in\Lambda$, 
we have $\theta^{\DC}_{\lambda}\in D^{(2)}(\DC_{\ell})$.
\end{Prop}
\begin{proof}
By Lemma \ref{symm}, we have 
$\theta^{\DC}_{\lambda}\in D^{(2)}(\AC_{\ell-1})$ for any 
$\lambda\in\Lambda$. 

Since 
\begin{align*}
&\left(\sum_{|\bm{\alpha}|=2}
s^{\DC}_{\lambda}\left(x_{\bm{\alpha}}\right)\frac{1}{\bm{\alpha}!}
\partial^{\bm{\alpha}}\right)\left((x_{i}+x_{j})\cdot x_{k}\right)\\
&=s^{\DC}_{\lambda}(x_{i},x_{k})+s^{\DC}_{\lambda}(x_{j},x_{k})\\
&=\frac{1}{x_{i}x_{k}}s^{\AC}_{\lambda}(x_{i}^2,x_{k}^2)+
\frac{1}{x_{j}x_{k}}s^{\AC}_{\lambda}(x_{j}^2,x_{k}^2)\\
&=\frac{1}{x_{i}x_{j}x_{k}}\left(x_{j}s^{\AC}_{\lambda}(x_{i}^2,x_{k}^2)+
x_{i}s^{\AC}_{\lambda}(x_{j}^2,x_{k}^2)\right),
\end{align*}
we obtain 
$\theta^{\DC}_{\lambda}((x_{i}+x_{j})x_{k})|_{x_{i}=-x_{j}}=0$ 
for $1\leq i<j\leq \ell$, $k=1,\dots,\ell$ and 
$\lambda\in\Lambda$. Hence we have 
$\theta^{\DC}_{\lambda}\in D^{(2)}(\DC)$ for any $\lambda\in\Lambda$.
\end{proof}
We introduce other operators $h^{\DC}_{k}$ of $D^{(2)}(\DC_{\ell})$. 
For $k=1,\dots,\ell$ put 
$h^{\DC}_{k}:=(x_{k}^2-x_{1}^2)\cdots(x_{k}^2-x_{k-1}^2)
(x_{k}^2-x_{k+1}^2)\cdots(x_{k}^2-x_{\ell}^2)$, 
and define 
$$
\eta^{\DC}_{k}:=\frac{h^{\DC}_{k}}{2x_{k}}\partial_{k}^2
-(-1)^{\ell-1}\frac{1}{x_{k}}\theta^{\DC}_{\lambda^{(0)}}.
$$
The coefficient of $\partial_{k}^2$ in $\eta^{\DC}_{k}$ is 
$$
\frac{h^{\DC}_{k}}{2x_{k}}-(-1)^{\ell-1}
\frac{(x_{1}\cdots x_{\ell})^2}{2x_{k}\cdot x_{k}^2}=
\frac{h^{\DC}_{k}-(-1)^{\ell-1}
(x_{1}\cdots x_{k-1}x_{k+1}\cdots x_{\ell})^2}{2x_{k}}
\in S.
$$
Hence we obtain $\eta^{\DC}_{k}\in D^{(2)}(S)$, 
and $\deg\eta^{\DC}_{k}=2\ell-2$.
\begin{Prop}\label{ope-d2}
For $k=1,\dots,\ell$, 
we have that $\eta^{\DC}_{k}\in D^{(2)}(\DC_{\ell})$.
\end{Prop}
\begin{proof}
Let $k=1,\dots,\ell$. It is clear that 
$\frac{h^{\DC}_{k}}{2}\partial_{k}^2\in D^{(2)}(\DC_{\ell})$ , and 
we have $\theta^{\DC}_{\lambda^{(0)}}\in D^{(2)}(\DC_{\ell})$ 
by Proposition \ref{ope-d1}. 
Thus we have $\frac{h^{\DC}_{k}}{2}\partial_{k}^2
-(-1)^{\ell-1}\theta^{\DC}_{\lambda^{(0)}}\in D^{(2)}(\DC_{\ell})$. 
This leads to $\eta^{\DC}_{k}\in D^{(2)}(\DC_{\ell})$.
\end{proof}
\begin{Them}\label{basisD}
Assume $m=2$. The set 
\begin{align*}
C_{\DC}&:=\left\{\eta^{\DC}_{i}\mid i=1,\dots\ell\right\}\cup
\left\{\theta^{\DC}_{\lambda}\mid\lambda\in\Lambda\right\}
\end{align*}
forms an $S$-basis for $D^{(2)}(\DC_{\ell})$. Hence 
\begin{align*}
\exp D^{(2)}(\DC_{\ell})=&\{2\ell-2,\dots,2\ell-2\}
\cup\{2\lambda_{1}+2\lambda_{2}-2\mid
\ell-2\geq\lambda_{1}\geq\lambda_{2}\geq 1\}\\
&\cup\{2\lambda_{1}-2+\ell\mid
\ell-2\geq\lambda_{1}\geq 1\}\cup\{2\ell-2\}.
\end{align*}
\end{Them}
\begin{proof}
By Proposition \ref{ope-d1} and Proposition \ref{ope-d2}, 
we have $C_{\DC}\subseteq D^{(2)}(\DC_{\ell})$. 
Let $M_{2}(C_{\DC})$ be the coefficient matrix of 
the operators in $C_{\DC}$. We shall show that 
$\det M_{2}(C_{\DC})\doteq Q(\DC_{\ell})^{\ell}$. 

Put $\theta_{\lambda}:=\sum_{|\bm{\alpha}|=2}
s^{\DC}_{\lambda}\left(x_{\bm{\alpha}}\right)\frac{1}{\bm{\alpha}!}
\partial^{\bm{\alpha}}$ for $\lambda\in\Lambda$. Then 
\begin{align*}
\det M_{2}(C_{\DC})&=\det M_{2}(\eta^{\DC}_{i},
\theta^{\DC}_{\lambda}\mid i=1,\dots,\ell,\ \lambda\in\Lambda)\\
&=\det M_{2}(\eta^{\DC}_{i}+(-1)^{\ell-1}
\frac{1}{x_{i}}\theta^{\DC}_{\lambda^{(0)}},
\theta^{\DC}_{\lambda}\mid i=1,\dots,\ell,\ \lambda\in\Lambda)\\
&=(x_{1}\cdots x_{\ell})^{\ell}\det M_{2}(\eta^{\DC}_{i}+(-1)^{\ell-1}
\frac{1}{x_{i}}\theta^{\DC}_{\lambda^{(0)}},
\theta_{\lambda}\mid i=1,\dots,\ell,\ \lambda\in\Lambda)\\
&\doteq\left(\frac{h^{\DC}_{1}}{x_{1}}\cdots 
\frac{h^{\DC}_{\ell}}{x_{\ell}}\right)(x_{1}\cdots x_{\ell})^{\ell}
\begin{vmatrix}
I_{\ell}&\ast\\
0&\det\left(s^{\DC}_{\lambda}(x_{\bm{\alpha}})\right)
_{\substack{\lambda\in\Lambda\\ \bm{\alpha}\in Z}}
\end{vmatrix}\\
&=Q(\DC_{\ell})^2 (x_{1}\cdots x_{\ell})^{\ell-1}
\frac{Q(\DC_{\ell})^{\ell-2}}{(x_{1}\cdots x_{\ell})^{\ell-1}}
=Q(\DC_{\ell})^{\ell}
\end{align*}
by Proposition \ref{det}. Hence we conclude that the set 
$C_{\DC}$ forms an $S$-basis for $D^{(2)}(\DC_{\ell})$ by 
Proposition \ref{Saito-Holm}.
\end{proof}
\section{Group actions}\label{GA}
Let $W$ be a finite reflection group generated by 
reflections acting on $V$. Then $W$ acts on $S$ by 
$(w\cdot f)(v)=f(w^{-1}\cdot v)$ for $f\in S$, 
$w\in W$ and $v\in V$. 
The action of $W$ on $D^{(m)}(S)$ is defined by 
$(w\cdot\theta)(f)=w\cdot(\theta(w^{-1}\cdot f))$ 
for $w\in W$, $\theta\in D^{(m)}(S)$ and $f\in S$. 

Let $\mathfrak{S}_{\ell}$ be the symmetric group 
acting on $V$ by permuting the coordinates. 
Let $\ZB/2\ZB=\{1,-1\}$. An abelian group $(\ZB/2\ZB)^{\ell}$
acts on $V$ by change of signs. 
Let $(\ZB/2\ZB)^{\ell-1}$ be the subgroup of $(\ZB/2\ZB)^{\ell}$ 
defined by 
$$
\left(\ZB/2\ZB\right)^{\ell-1}=
\left\{(a_{1},\dots,a_{\ell})\in(\ZB/2\ZB)^{\ell}\mid
a_{1}\cdots a_{\ell}=1\right\}.
$$
The group $\mathfrak{S}_{\ell}$ acts on 
$(\ZB/2\ZB)^{\ell}$ and $(\ZB/2\ZB)^{\ell-1}$ by 
permuting the coodinates.

The finite irreducible reflection groups of 
types A, B and D are defined by 
\begin{align*}
W^{A}&:=\mathfrak{S}_{\ell},\\
W^{B}&:=\mathfrak{S}_{\ell}\ltimes(\ZB/2\ZB)^{\ell},\\
W^{D}&:=\mathfrak{S}_{\ell}\ltimes(\ZB/2\ZB)^{\ell-1}.
\end{align*}
From now on, we assume $\ell\geq 4$ when we consider the 
reflection group of type D. 
Then the groups $W^{A}$, $W^{B}$ and $W^{D}$ act on $V$. 
Hence the groups $W^{A}$, $W^{B}$ and $W^{D}$ act on 
$S$ and $D^{(m)}(S)$.
\begin{Prop}\label{action}
Let $W$ be a finite reflection group, and $\AS$ 
the reflection arrangement consisting of 
all reflection hyperplanes of $W$. 
Then the submodule $D^{(m)}(\AS)$ of $D^{(m)}(S)$ 
is closed under the action of $W$.
\end{Prop}
\begin{proof}
For $w\in W$ and $\theta\in D^{(m)}(\AS)$, 
we prove that $w\cdot\theta\in D^{(m)}(\AS)$. 

For $f\in S$, we have 
\begin{align*}
w\cdot\theta(Qf)&=w\cdot(\theta(w^{-1}\cdot(Qf)))\\
&=w\cdot(\theta(\det(w^{-1})Q\,(w^{-1}\cdot f))).
\end{align*}
Since $\det(w^{-1})Q\,(w^{-1}\cdot f)\in QS$, we have 
$\theta(\det(w^{-1})Q\,(w^{-1}\cdot f))\in QS$. 
Then there exists $g\in S$ such that 
$\theta(\det(w^{-1})Q\,(w^{-1}\cdot f))=Qg$. 
Hence
\begin{align*}
w\cdot\theta(Qf)=w\cdot(Qg)=(\det(w)Q)(w\cdot g)\in QS.
\end{align*}
\end{proof}
By Proposition \ref{action}, the groups 
$W^{A}$, $W^{B}$ and $W^{D}$ act on 
$D^{(m)}(\AC)$, $D^{(m)}(\BC)$ and $D^{(m)}(\DC)$, respectively.

In case $m=1$, the modules $D^{(1)}(\AC)$, $D^{(1)}(\BC)$ 
and $D^{(1)}(\DC)$ have bases consisting of only 
invariant elements \cite[Theorem 6.60]{O-T}. 
In this section, we prove that $D^{(2)}(\AC)$, $D^{(2)}(\BC)$ 
and $D^{(2)}(\DC)$ cannot have bases consisting of 
only invariant elements, when $m=2$. 

The actions of a transposition 
$\sigma_{i,j}:=(i\,j)\in\mathfrak{S}_{\ell}$ on 
$\{x_{1},\dots x_{\ell}\}$ and 
$\{\partial_{1},\dots,\partial_{\ell}\}$ are as follows:
\begin{align*}
\sigma_{i,j}\cdot x_{k}=x_{\sigma_{i,j}\cdot k},\ 
\sigma_{i,j}\cdot \partial_{k}=\partial_{\sigma_{i,j}\cdot k}\qquad
(k=1,\dots,\ell).
\end{align*}
The group $\mathfrak{S}_{\ell}$ acts on 
the set of multi-indices by permuting the coordinates: 
$$
\sigma_{i,j}\cdot(\alpha_{1},\dots,\alpha_{i},\dots,\alpha_{j},\dots,\alpha_{\ell})
=(\alpha_{1},\dots,\alpha_{j},\dots,\alpha_{i},\dots,\alpha_{\ell})
$$
The action of $\mathfrak{S}_{\ell}$ preserves the norm of a multi-index. 
Then $\sigma_{i,j}\cdot x^{\bm{\alpha}}=x^{\sigma_{i,j}\cdot\bm{\alpha}}$ 
for a multi-index $\bm{\alpha}\in \NB^{\ell}$. 

Let $\tau_{i}\in(\ZB/2\ZB)^{\ell}$ be the element of 
change of signs of the $i$-th coordinate:
\begin{align}
\tau_{i}\cdot x_{k}=a_{i,k}x_{k},\ 
\tau_{i}\cdot \partial_{k}=a_{i,k}\partial_{k}\qquad
(k=1,\dots,\ell)\label{acttau}
\end{align}
where $a_{i,i}=-1$ and $a_{i,k}=1$ for $k\neq i$. 
\begin{Lem}\label{inv-theta}
Let $\lambda\in\Lambda$.
\begin{enumerate}
\item[$(1)$]
The operator $\theta^{\AC}_{\lambda}$ is $W^{A}$-invariant.
\item[$(2)$]
The operator $\theta^{\BC}_{\lambda}$ is $W^{B}$-invariant.
\item[$(3)$]
The operator $\theta^{\DC}_{\lambda}$ is $W^{D}$-invariant when $m=2$.
\end{enumerate}
\end{Lem}
\begin{proof}
$(1)$ Since $W^{A}$ is generated by transpositions 
$\sigma_{1,2},\dots,\sigma_{\ell-1,\ell}$ 
(see, for example, \cite{Bou}), it is enough to prove that 
$$
\theta^{\AC}_{\lambda}=\sum_{|\bm{\alpha}|=m}
s^{\AC}_{\lambda}\left(x_{\bm{\alpha}}\right)\frac{1}{\bm{\alpha}!}
\partial^{\bm{\alpha}}
$$ 
is invariant under the actions of transpositions 
$\sigma_{1,2},\dots,\sigma_{\ell-1,\ell}$. 

Clearly, we have that $(\sigma_{i,i+1}\cdot\bm{\alpha})!=\bm{\alpha}!$ 
and $|\sigma_{i,i+1}\cdot\bm{\alpha}|=|\bm{\alpha}|$ for a multi-index 
$\bm{\alpha}$ with $|\bm{\alpha}|=m$. 
Then a transposition $\sigma_{i,i+1}$ is a bijection between 
the set $\{\bm{\alpha}\in\NB^{\ell}\mid |\bm{\alpha}|=m\}$ 
and itself. Therefore we have that, for $i=1,\dots,\ell-1$, 
\begin{align*}
\sigma_{i,i+1}\cdot\theta^{\AC}_{\lambda}&=
\sum_{|\sigma_{i,i+1}\cdot\bm{\alpha}|=m}\sigma_{i,i+1}\cdot\left(
s^{\AC}_{\lambda}(x_{\sigma_{i,i+1}\cdot\bm{\alpha}})\right)
\frac{1}{(\sigma_{i,i+1}\cdot\bm{\alpha})!}
\sigma_{i,i+1}\cdot\left(\partial^{\sigma_{i,i+1}\cdot\bm{\alpha}}\right)\\
&=\sum_{|\sigma_{i,i+1}\cdot\bm{\alpha}|=m}
s^{\AC}_{\lambda}(x_{\sigma_{i,i+1}\cdot\sigma_{i,i+1}\cdot\bm{\alpha}})
\frac{1}{(\sigma_{i,i+1}\cdot\bm{\alpha})!}
\partial^{\sigma_{i,i+1}\cdot\sigma_{i,i+1}\cdot\bm{\alpha}}\\
&=\sum_{|\bm{\alpha}|=m}s^{\AC}_{\lambda}
\left(x_{\bm{\alpha}}\right)\frac{1}{\bm{\alpha}!}
\cdot\partial^{\bm{\alpha}}=\theta^{\AC}_{\lambda}.
\end{align*}

$(2)$ The group $W^{B}$ is generated by $\tau_{\ell}$ and 
transpositions $\sigma_{1,2},\dots,\sigma_{\ell-1,\ell}$ 
(see \cite{Bou}). 
It is enough to prove that $\theta^{\BC}_{\lambda}$ 
is invariant under the actions of the generators. 

By the formulas (\ref{x_a^2}) and (\ref{exprB}), we have 
\begin{align}
s^{\BC}_{\lambda}(x_{\bm{\alpha}})=
x_{1}^{\alpha_{1}}\cdots x_{\ell}^{\alpha_{\ell}}
s^{\AC}_{\lambda}(x_{\bm{\alpha}}^2)\label{sx_a}
\end{align}
for a multi-index $\bm{\alpha}$ with $|\bm{\alpha}|=m$. 
Then 
\begin{align*}
\sigma_{i,i+1}\cdot s^{\BC}_{\lambda}(x_{\bm{\alpha}})&=
x_{\sigma_{i,i+1}\cdot 1}^{\alpha_{1}}\cdots
x_{\sigma_{i,i+1}\cdot\ell}^{\alpha_{\ell}}
s^{\AC}_{\lambda}(x_{\sigma_{i,i+1}\cdot\bm{\alpha}}^2)\\
&=x_{1}^{\alpha_{\sigma_{i,i+1}\cdot 1}}\cdots
x_{\ell}^{\alpha_{\sigma_{i,i+1}\cdot \ell}}
s^{\AC}_{\lambda}(x_{\sigma_{i,i+1}\cdot\bm{\alpha}}^2)=
s^{\BC}_{\lambda}(x_{\sigma_{i,i+1}\cdot\bm{\alpha}}).
\end{align*}
Hence we have that 
\begin{align*}
\sigma_{i,i+1}\cdot\theta^{\BC}_{\lambda}
=\sum_{|\bm{\alpha}|=m}s^{\BC}_{\lambda}
\left(x_{\sigma_{i,i+1}\cdot\bm{\alpha}}\right)
\frac{1}{(\sigma_{i,i+1}\cdot\bm{\alpha})!}
\cdot\partial^{\sigma_{i,i+1}\cdot\bm{\alpha}}
=\theta^{\BC}_{\lambda}.
\end{align*}
It remains to prove that $\theta^{\BC}_{\lambda}$ is 
$\tau_{\ell}$-invariant. 

By the formulas (\ref{acttau}) and (\ref{sx_a}), we have 
$\tau_{\ell}\cdot\partial^{\bm{\alpha}}=
(-1)^{\alpha_{\ell}}\partial^{\bm{\alpha}}$ and 
$\tau_{\ell}\cdot s^{\BC}_{\lambda}(x_{\bm{\alpha}})=
(-1)^{\alpha_{\ell}} s^{\BC}_{\lambda}(x_{\bm{\alpha}})$. 
Then we have 
\begin{align*}
\tau_{\ell}\cdot\theta^{\BC}_{\lambda}=
\sum_{|\bm{\alpha}|=m}(-1)^{\alpha_{\ell}}s^{\BC}_{\lambda}
\left(x_{\bm{\alpha}}\right)\frac{1}{\bm{\alpha}!}
(-1)^{\alpha_{\ell}}\partial^{\tau_{\ell}\cdot\bm{\alpha}}
=\theta^{\BC}_{\lambda}.
\end{align*}
Hence the operator $w\cdot\theta^{\BC}_{\lambda}$ coincides 
with $\theta^{\BC}_{\lambda}$ for $w\in W^{B}$.

$(3)$ The proof of $(3)$ goes similarly to $(2)$. 
The set 
$$
\{\sigma_{1,2},\dots,\sigma_{\ell-1,\ell}
,\sigma_{\ell-1,\ell}\tau_{\ell-1}\tau_{\ell}\}
$$ 
is a system of generators of $W^{D}$ (see \cite{Bou}).

The formulas (\ref{x_a^2}) and (\ref{exprD}) imply that
\begin{align}
s^{\DC}_{\lambda}(x_{\bm{\alpha}})=
\frac{1}{x_{1}^{\alpha_{1}}\cdots x_{\ell}^{\alpha_{\ell}}}
s^{\AC}_{\lambda}(x_{\sigma_{i,i+1}\cdot\bm{\alpha}}^2).\label{sDx_a}
\end{align}
Then the formula (\ref{sDx_a}) leads to 
\begin{align*}
\sigma_{i,i+1}\cdot s^{\DC}_{\lambda}(x_{\bm{\alpha}})=
\frac{1}{\sigma_{i,i+1}\cdot
(x_{1}^{\alpha_{1}}\cdots x_{\ell}^{\alpha_{\ell}})}
s^{\AC}_{\lambda}(x_{\sigma_{i,i+1}\cdot\bm{\alpha}}^2)=
s^{\DC}_{\lambda}(x_{\sigma_{i,i+1}\cdot\bm{\alpha}}).
\end{align*}
We obtain 
$\sigma_{i,i+1}\cdot\theta^{\DC}_{\lambda}=\theta^{\DC}_{\lambda}$ 
by a straightforward calculation using the formulas 
(\ref{thetaD1}), (\ref{thetaD2}) and (\ref{thetaD3}).

In order to verify 
$\sigma_{\ell-1,\ell}\tau_{\ell-1}\tau_{\ell}
\cdot\theta^{\DC}_{\lambda}=\theta^{\DC}_{\lambda}$, 
we compute the actions of $\sigma_{1,\ell-1}\tau_{\ell-1}\tau_{\ell}$ 
on polynomials and differential operators:
\begin{align*}
\sigma_{\ell-1,\ell}\tau_{\ell-1}\tau_{\ell}\cdot
s^{\DC}_{\lambda}(x_{\bm{\alpha}})&=(-1)^{\alpha_{\ell-1}+\alpha_{\ell}}
\frac{1}{\sigma_{\ell-1,\ell}\cdot(x_{1}^{\alpha_{1}}\cdots
x_{\ell}^{\alpha_{\ell}})}s^{\AC}_{\lambda}
(x_{\sigma_{\ell-1,\ell}\cdot\bm{\alpha}}^2)\\
&=(-1)^{\alpha_{\ell-1}+\alpha_{\ell}}s^{\DC}_{\lambda}(x_{\bm{\alpha}}),\\
\sigma_{\ell-1,\ell}\tau_{\ell-1}\tau_{\ell}\cdot\partial^{\bm{\alpha}}
&=(-1)^{\alpha_{\ell-1}+\alpha_{\ell}}
\partial^{\sigma_{\ell-1,\ell}\cdot\bm{\alpha}},\\
\sigma_{\ell-1,\ell}\tau_{\ell-1}\tau_{\ell}\cdot
\frac{1}{x_{1}\cdots x_{\ell}}&=\frac{1}{x_{1}\cdots x_{\ell}},\\
\sigma_{\ell-1,\ell}\tau_{\ell-1}\tau_{\ell}\cdot
\frac{1}{(x_{1}\cdots x_{\ell})^2}&=\frac{1}{(x_{1}\cdots x_{\ell})^2}.
\end{align*}
By the case-by-case checking, we can verify that 
$\sigma_{\ell-1,\ell}\tau_{\ell-1}\tau_{\ell}
\cdot\theta^{\DC}_{\lambda}$ coincides with $\theta^{\DC}_{\lambda}$
for $\lambda\in\Lambda$. 
Therefore the operator $\theta^{\DC}_{\lambda}$ is $W^{D}$-invariant.
\end{proof}
\begin{Lem}\label{eta-V}
\begin{enumerate}
\item[$(1)$] The vector space spanned by 
$\left\{\eta^{\AC}_{k}\mid k=1,\dots\ell\right\}$ 
is closed under the action of $W^{A}$. 
Moreover, the vector space spanned by 
$\left\{\eta^{\AC}_{k}\mid k=1,\dots\ell\right\}$ 
is isomorphic to the Euclidean space $V$ as $W^A$-modules.
\item[$(2)$] The vector space spanned by 
$\left\{\eta^{\BC}_{k}\mid k=1,\dots\ell\right\}$ 
is closed under the action of $W^{B}$. 
Moreover, the vector space spanned by 
$\left\{\eta^{\BC}_{k}\mid k=1,\dots\ell\right\}$ 
is isomorphic to the Euclidean space $V$ as $W^B$-modules.
\item[$(3)$] Let $m=2$. The vector space spanned by 
$\left\{\eta^{\DC}_{k}\mid k=1,\dots\ell\right\}$ 
is closed under the action of $W^{D}$. 
Moreover, the vector space spanned by 
$\left\{\eta^{\DC}_{k}\mid k=1,\dots\ell\right\}$ 
is isomorphic to the Euclidean space $V$ as $W^D$-modules.
\end{enumerate}
\end{Lem}
\begin{proof}
$(1)$ Let $\langle\eta^{\AC}_{1},\dots,\eta^{\AC}_{\ell}\rangle_{\RB}$ 
be the vector space spanned by 
$\left\{\eta^{\DC}_{k}\mid k=1,\dots\ell\right\}$. 
Define a linear isomorphism 
$$
\phi_{A}\,:\,
\langle\eta^{\AC}_{1},\dots,\eta^{\AC}_{\ell}\rangle_{\RB}
\longrightarrow V
$$
by $\phi_{A}(\eta^{\AC}_{k})=e_{k}$ for $k=1,\dots,\ell$. 

A transposition $\sigma_{i,i+1}$ acts on the standard basis 
$e_{1},\dots,e_{\ell}$ of $V$ by 
$\sigma_{i,i+1}\cdot e_{k}=e_{\sigma_{i,i+1}\cdot k}$ 
for $k=1,\dots,\ell$. 
To prove the assertion, we verify 
$\sigma_{i,i+1}\cdot\eta^{\AC}_{k}=\eta^{\AC}_{\sigma_{i,i+1}\cdot k}$ 
for $k=1,\dots,\ell$. 

Recall that the definitions 
$h^{\AC}_{k}=(x_{k}-x_{1})\cdots(x_{k}-x_{k-1})
(x_{k}-x_{k+1})\cdots(x_{k}-x_{\ell})$ and 
$\eta_{k}^{\AC}=h^{\AC}_{k}\frac{1}{m!}\partial_{k}^m$. 
For $k\neq i,i+1$, we have 
$\sigma_{i,i+1}\cdot h^{\AC}_{k}=h^{\AC}_{k}$. 
Also we have 
$$
\sigma_{i,i+1}\cdot h^{\AC}_{i}=
(x_{i+1}-x_{1})\cdots(x_{i+1}-x_{i-1})
(x_{i+1}-x_{i})(x_{i+1}-x_{i+2})\cdots(x_{k}-x_{\ell})=h^{\AC}_{i+1},
$$
and similarly $\sigma_{i,i+1}\cdot h^{\AC}_{i+1}=h^{\AC}_{i}$. 
Then we obtain 
$$
\sigma_{i,i+1}\cdot\eta^{\AC}_{k}=\eta^{\AC}_{k}\ (k\neq i,i+1),\ 
\sigma_{i,i+1}\cdot\eta^{\AC}_{i}=\eta^{\AC}_{i+1},\ 
\sigma_{i,i+1}\cdot\eta^{\AC}_{i+1}=\eta^{\AC}_{i}.
$$
Hence we conclude that the map $\phi_{A}$ is a $W^A$-isomorphism.

$(2)$ A system of generators 
$\sigma_{1,2},\dots,\sigma_{\ell-1,\ell}$ and $\tau_{\ell}$ 
of $W^{B}$ acts on the standard basis 
$e_{1},\dots,e_{\ell}$ of $V$ by 
$\sigma_{i,i+1}\cdot e_{k}=e_{\sigma_{i,i+1}\cdot k}$ 
for $k=1,\dots,\ell$ and 
$$
\tau_{\ell}\cdot e_{\ell}=-e_{\ell},\ 
\tau_{\ell}\cdot e_{k}=e_{k}\ (k=1,\dots,\ell-1).
$$
Define a linear isomorphism 
$$
\phi_{B}\,:\,
\langle\eta^{\BC}_{1},\dots,\eta^{\BC}_{\ell}\rangle_{\RB}
\longrightarrow V
$$
by $\phi_{B}(\eta^{\BC}_{k})=e_{k}$ for $k=1,\dots,\ell$. 
We prove that $\phi_{B}$ is a $W^B$-isomorphism 
by checking all actions of the generators are the same. 

Recall that the definitions 
$h^{\BC}_{k}=x_{k}(x_{k}^2-x_{1}^2)\cdots(x_{k}^2-x_{k-1}^2)
(x_{k}^2-x_{k+1}^2)\cdots(x_{k}^2-x_{\ell}^2)$ and 
$\eta_{k}^{\BC}=h^{\BC}_{k}\frac{1}{m!}\partial_{k}^m$. 
We have the following: 
\begin{align*}
\sigma_{i,i+1}\cdot h^{\BC}_{k}&=h^{\BC}_{k}\qquad
(k\neq i,i+1),\\
\sigma_{i,i+1}\cdot h^{\BC}_{i}&=h^{\BC}_{i+1},\\
\sigma_{i,i+1}\cdot h^{\BC}_{i+1}&=h^{\BC}_{i},\\
\tau_{\ell}\cdot h^{\BC}_{k}&=h^{\BC}_{k}
\qquad (k\neq \ell),\\
\tau_{\ell}\cdot h^{\BC}_{\ell}&=-h^{\BC}_{\ell}.
\end{align*}
Then
\begin{align*}
\sigma_{i,i+1}\cdot\eta^{\BC}_{k}=\eta^{\BC}_{k}\ &(k\neq i,i+1),\ 
\sigma_{i,i+1}\cdot\eta^{\BC}_{i}=\eta^{\BC}_{i+1},\ 
\sigma_{i,i+1}\cdot\eta^{\BC}_{i+1}=\eta^{\BC}_{i},\\
&\tau_{\ell}\cdot\eta^{\BC}_{k}=\eta^{\BC}_{k}\ (k\neq \ell),\ 
\tau_{\ell}\cdot\eta^{\BC}_{\ell}=\eta^{\BC}_{\ell}.
\end{align*}
The actions of the generators on 
$\{\eta^{\BC}_{k}\mid k=1,\dots,\ell\}$ coincide with the actions of 
the generators on $\{e_{k}\mid k=1,\dots,\ell\}$. 
Hence the map $\phi_{B}$ is a $W^B$-isomorphism.

$(3)$ The group $W^{D}$ is generated by elements 
$\sigma_{1,2},\dots,\sigma_{\ell-1,\ell}$ and 
$\sigma_{\ell-1,\ell}\tau_{\ell-1}\tau_{\ell}$. 
The generators acts on the standard basis 
$e_{1},\dots,e_{\ell}$ of $V$ by 
$\sigma_{i,i+1}\cdot e_{k}=e_{\sigma_{i,i+1}\cdot k}$ 
for $k=1,\dots,\ell$ and 
\begin{align*}
\sigma_{\ell-1,\ell}&\tau_{\ell-1}\tau_{\ell}\cdot 
e_{\ell}=-e_{\ell-1},\ 
\sigma_{\ell-1,\ell}\tau_{\ell-1}\tau_{\ell}\cdot 
e_{\ell-1}=-e_{\ell},\\
&\sigma_{\ell-1,\ell}\tau_{\ell-1}\tau_{\ell}\cdot
 e_{k}=e_{k}\qquad(k=1,\dots,\ell-2).
\end{align*}
We prove that a linear isomorphism 
$$
\phi_{D}\,:\,
\langle\eta^{\DC}_{1},\dots,\eta^{\DC}_{\ell}\rangle_{\RB}
\longrightarrow V
$$
defined by $\phi_{D}(\eta^{\DC}_{k})=e_{k}$ for $k=1,\dots,\ell$ 
is a $W^{D}$-isomorphism. 

Recall that the definitions 
$h^{\DC}_{k}=(x_{k}^2-x_{1}^2)\cdots(x_{k}^2-x_{k-1}^2)
(x_{k}^2-x_{k+1}^2)\cdots(x_{k}^2-x_{\ell}^2)$ and 
$\eta^{\DC}_{k}=\frac{h^{\DC}_{k}}{2x_{k}}\partial_{k}^2
-(-1)^{\ell-1}\frac{1}{x_{k}}\theta^{\DC}_{\lambda^{(0)}}$. 
Clearly we have 
$$
\sigma_{i,i+1}\cdot h^{\DC}_{i}=h^{\DC}_{i+1},\ 
\sigma_{i,i+1}\cdot h^{\DC}_{i+1}=h^{\DC}_{i},\ 
\sigma_{i,i+1}\cdot h^{\DC}_{k}=h^{\DC}_{k}\ (k\neq i,i+1),
$$
and then 
$$
\sigma_{i,i+1}\cdot \eta^{\DC}_{i}=\eta^{\DC}_{i+1},\ 
\sigma_{i,i+1}\cdot \eta^{\DC}_{i+1}=\eta^{\DC}_{i},\ 
\sigma_{i,i+1}\cdot \eta^{\DC}_{k}=\eta^{\DC}_{k}\ (k\neq i,i+1)
$$
by Lemma \ref{inv-theta}. 
The actions of $\sigma_{\ell-1,\ell}\tau_{\ell-1}\tau_{\ell}$ 
on $\{h^{\DC}_{k}\mid k=1,\dots,\ell\}$ are 
\begin{align*}
\sigma_{\ell-1,\ell}&\tau_{\ell-1}\tau_{\ell}\cdot 
h^{\DC}_{\ell}=h^{\DC}_{\ell-1},\ 
\sigma_{\ell-1,\ell}\tau_{\ell-1}\tau_{\ell}\cdot 
h^{\DC}_{\ell-1}=h^{\DC}_{\ell},\\
&\sigma_{\ell-1,\ell}\tau_{\ell-1}\tau_{\ell}\cdot
h^{\DC}_{k}=h^{\DC}_{k}\qquad(k=1,\dots,\ell-2).
\end{align*}
Then it follows from Lemma \ref{inv-theta} that 
\begin{align*}
\sigma_{\ell-1,\ell}\tau_{\ell-1}\tau_{\ell}\cdot 
\eta^{\DC}_{\ell}&=-\frac{h^{\DC}_{\ell-1}}{2x_{\ell-1}}\partial_{\ell-1}^2
+(-1)^{\ell-1}\frac{1}{x_{\ell-1}}\theta^{\DC}_{\lambda^{(0)}}=
-\eta^{\DC}_{\ell-1},\\
\sigma_{\ell-1,\ell}\tau_{\ell-1}\tau_{\ell}\cdot 
\eta^{\DC}_{\ell-1}&=-\frac{h^{\DC}_{\ell}}{2x_{\ell}}\partial_{\ell}^2
+(-1)^{\ell-1}\frac{1}{x_{\ell}}\theta^{\DC}_{\lambda^{(0)}}=
-\eta^{\DC}_{\ell},\\
\sigma_{\ell-1,\ell}\tau_{\ell-1}\tau_{\ell}\cdot
\eta^{\DC}_{k}&=\eta^{\DC}_{k}\quad(k=1,\dots,\ell-2).
\end{align*}
Thus the actions of the generators on 
$\{\eta^{\DC}_{k}\mid k=1,\dots,\ell\}$ coincide with the actions of 
the generators on $\{e_{k}\mid k=1,\dots,\ell\}$. 
Hence the map $\phi_{D}$ is a $W^D$-isomorphism.
\end{proof}
\begin{Cor}\label{repofbasis}
Assume that $m=2$. 
Let $W$ be a finite reflection group of type A, B or D. 
Let $\AS$ be the reflection arrangement corresponding to $W$. 
Then the vector space $X$ generated by the basis $C_{\AS}$ 
in Theorem \ref{basisAB} or Theorem \ref{basisD} 
is a $W$-module, and $X$ is isomorphic as a representeation to 
$X_{0}^{\sharp\Lambda}\oplus V$ 
where $X_{0}$ is the trivial representation.
\end{Cor}
In the case of type $A_{\ell-1}$, the reflection group $W^{A}$ 
stabilizes the subspace $\RB(e_{1}+\cdots+e_{\ell})$ 
of $V$ pointwisely. The orthogonal complement $V^{\prime}$ of 
$\RB(e_{1}+\cdots+e_{\ell})$ is closed under the action of 
the reflection group $W^{A}$, and $V^{\prime}$ is essensial 
(see \cite{Hum}). 
Moreover $V^{\prime}$ is an irreducible representation. 
In the case of type B or D, a representation $V$ 
is irreducible (see Bourbaki \cite[Chap. 5, Sect. 3, Proposition 5]{Bou}). 

Let $D^{(m)}(S)^{W}$ be the set of invariant elements 
of $D^{(m)}(S)$. 
\begin{Them}
Assume that $m=2$. 
Let $W$ be a finite reflection group of type A, B or D, 
and let $\AS$ be the reflection arrangement corresponding to $W$. 
Then The module $D^{(2)}(\AS)$ cannot have bases 
consisting of only invariant elements.
\end{Them}
\begin{proof}
Since proofs of type B and D are almost the same with the proof of type A, 
we prove the assertion only in the case of type A. 

We assume that $\AS=\AC_{\ell-1}$ and $W=W^{A}$. 
Let $\theta_{1}^{\prime},\dots,\theta_{s_{2}}^{\prime}$ be an 
$S$-basis for $D^{(2)}(\AS)$. Since the degrees 
do not depend on a choice of a basis, we have 
$\{\deg\theta_{1}^{\prime},\dots,\deg\theta_{s_{2}}^{\prime}\}
=\exp D^{(2)}(\AC_{\ell-1})$ by Theorem \ref{basisAB}. 
Assume 
$\deg\theta_{1}^{\prime}\leq\cdots\leq\deg\theta_{s_{2}}^{\prime}$. 
We show that there exists $\theta_{j}^{\prime}$ with 
$\deg\theta_{j}^{\prime}=\ell-1$ (replace $\ell-1$ by $2\ell-1$ 
and by $2\ell-2$ in the case type B and D, respectively) such that 
$\theta_{j}^{\prime}\notin D^{(2)}(S)^W$. 

Suppose that $\theta_{j}^{\prime}\in D^{(2)}(S)^W$ 
for any $j$. Since $C_{\AC}$ is a basis for $D^{(2)}(\AS)$ by 
Theorem \ref{basisAB}, we may write 
\begin{align*}
\theta_{j}^{\prime}=\sum_{\lambda}f_{\lambda}\theta^{\AC}_{\lambda}
+\sum_{k=1}^{\ell}a_{k}\eta_{k}^{\AC}
\end{align*}
for some $f_{\lambda}\in S$ and $a_{k}\in\RB$. 
For any $w\in W$, we have 
\begin{align*}
w\cdot\theta_{j}^{\prime}=\sum_{\lambda}(w\cdot 
f_{\lambda})\theta^{\AC}_{\lambda}
+w\cdot\sum_{k=1}^{\ell}a_{k}\eta_{k}^{\AC}
\end{align*}
by Lemma \ref{inv-theta}. Since 
$w\cdot\theta_{j}^{\prime}=\theta_{j}^{\prime}$ and 
$C_{\AC}$ is linearly independent over $S$, 
we have that $f_{\lambda}$ is $W$-invariant for 
$\lambda\in\Lambda$. Then 
$$
\sum_{k=1}^{\ell}a_{k}\eta_{k}^{\AC}\in D^{(2)}(S)^W.
$$

By Lemma \ref{eta-V}, the vector space 
$\langle\eta^{\AC}_{1},\dots,\eta^{\AC}_{\ell}\rangle_{\RB}
/\RB(\eta^{\AC}_{1}+\cdots+\eta^{\AC}_{\ell})$ is 
a nontrivial irreducible representation. Thus 
we have $a_{1}=\cdots=a_{\ell}$ 
(replace it by $a_{1}=0,\dots,a_{\ell}=0$ in the case of 
type B and D). 
This leads that $\theta_{1}^{\prime},\dots,\theta_{s_{2}}^{\prime}$ 
is linearly dependent over $S$. This is a contradiction. 
\end{proof}
\section*{Acknowledgements}
The author thanks Professor Soichi Okada
 for his helpful comments.

\end{document}